\documentclass[11pt]{amsart}
\usepackage{amssymb,mathrsfs,stmaryrd}
\usepackage{xcolor}
\newtheorem{theorem}{Theorem}
 
\newtheorem{proposition}[theorem]{Proposition} 
\newtheorem{corollary}[theorem]{Corollary} 
\theoremstyle{definition}

\newtheorem{conjecture}[theorem]{Conjecture}

\begin{document}
\title[Minimal Two-Spheres]{Minimal Two-Spheres and Manifolds with Positive Isotropic Curvature}

\author{Tsz-Kiu Aaron Chow}
\address{Department of Mathematics, Hong Kong University of Science and Technology, Hong Kong S.A.R., China}
\email{chowtka@ust.hk}

\author{Yipeng Wang}
\address{Department of Mathematics, Princeton University, Princeton, NJ 08544, USA}
\email{yipeng.wang@princeton.edu}

\thanks{We thank Otis Chodosh for his interest in this work and for pointing out the relation between positive isotropic curvature and positive \(4\)-intermediate curvature in dimension six. T.-K. A. C. is supported  by the Croucher Foundation Start-up Grant and the HKUST New Faculty Start-up Grant.}

\begin{abstract}
We improve the Micallef--Moore index estimate for harmonic two-spheres in $n$-manifolds with positive isotropic curvature to the sharp bound $n-3$. Combining with recent work, this completes the diffeomorphism classification of closed manifolds with positive isotropic curvature in the remaining dimensions five and six.
\end{abstract}

\maketitle

\section{Introduction}
Let $(M,g)$ be a Riemannian manifold of dimension $n\geq 4$. For $p\in M$, we define $T_p^{\mathbb{C}}M=T_pM\otimes _{\mathbb{R}}\mathbb{C}$ to be the complexified tangent space to $M$ at $p$. The inner product $g$ extends to a complex bilinear form $g:T_p^{\mathbb{C}}M\times T_p^{\mathbb{C}}M\to \mathbb{C}$, and the Riemann curvature tensor $R$ extends to a complex multilinear form $R:T_p^{\mathbb{C}}M\times T_p^{\mathbb{C}}M\times T_p^{\mathbb{C}}M\times T_p^{\mathbb{C}}M\to \mathbb{C}$. We say $(M,g)$ has positive isotropic curvature if and only if $R(\zeta,\eta,\bar\zeta,\bar\eta)>0$ for every $p\in M$ and all linearly independent $\zeta,\eta\in T_p^{\mathbb C}M$ satisfying $g(\zeta,\zeta)=g(\zeta,\eta)=g(\eta,\eta)=0$. Equivalently, $(M,g)$ has positive isotropic curvature if for every $p\in M$ and every orthonormal four-frame $\{e_1,e_2,e_3,e_4\}$, we have
\[
R_{1313}+R_{1414}+R_{2323}+R_{2424}-2R_{1234}>0,
\]
where $R_{ijk\ell}=R(e_i,e_j,e_k,e_{\ell})$.

We identify $S^2=\mathbb{R}^2\cup\{\infty\}$ via the stereographic projection. Let $f:S^2\to M$ be a non-constant harmonic map, i.e. $f$ is a non-constant critical point of the energy functional
\[
\mathscr{E}(f)=\frac{1}{2}\int_{S^2}\Big(\Big|\frac{\partial f}{\partial x}\Big|^2+\Big|\frac{\partial f}{\partial y}\Big|^2\Big)\,dx\,dy.
\]
In particular, this is equivalent to $D_{\frac{\partial}{\partial x}}\frac{\partial f}{\partial x}+D_{\frac{\partial}{\partial y}}\frac{\partial f}{\partial y}=0$. Moreover, we note that the classical Hopf differential argument implies $f$ is a branched minimal immersion.

Let $I:\Gamma(f^*(TM))\times \Gamma(f^*(TM))\to \mathbb{R}$ be the second variation of $\mathscr{E}$ given by
\begin{align*}
    I(V,V)&=\int_{S^2}\Big[|D_{\frac{\partial}{\partial x}}V|^2+|D_{\frac{\partial}{\partial y}}V|^2\Big]\,dx\,dy\\
    &-\int_{S^2}\Big[R\Big(\frac{\partial f}{\partial x},V,\frac{\partial f}{\partial x},V\Big)+R\Big(\frac{\partial f}{\partial y},V,\frac{\partial f}{\partial y},V\Big)\Big]\,dx\,dy.
\end{align*}
Let $\mathscr N\subset\Gamma(f^*(TM))$ be the negative eigenspace of the Jacobi
operator associated with $I$. The Morse index of $f$ is defined by $\dim_{\mathbb{R}}\mathscr N$. The main result of the manuscript is the following.
\begin{theorem}\label{thm}
    Let $(M,g)$ be a  Riemannian manifold of dimension $n\geq 4$ with positive isotropic curvature. Then, every non-constant harmonic map $f:S^2\to M$ has Morse index at least $n-3$. 
\end{theorem}
As a consequence of the classical Sacks--Uhlenbeck existence theory of harmonic maps \cite{SacksUhlenbeck1981}, we obtain the following improved vanishing result for homotopy groups of closed manifolds with positive isotropic curvature.
\begin{corollary}\label{cor:vanishing}
    Let $(M,g)$ be a closed Riemannian manifold of dimension $n\geq 4$ with positive isotropic curvature. Then $\pi_k(M)=0$ for all $k\in \{2,\cdots,n-2\}$.
\end{corollary}
Indeed, if $\pi_k(M)\neq0$ for $k\geq2$, then by \cite[Chapter VII, Theorem 2]{SchoenYau1997}, there exists a non-constant harmonic map $f:S^2\to M$ with Morse index at most $k-2$, and hence Theorem \ref{thm} implies $k\geq n-1$.
In their classical work \cite{MM1988}, Micallef and Moore showed that under the assumptions of Theorem \ref{thm}, every non-constant harmonic map $f:S^2\to M$ has Morse index at least $\lfloor\frac{n-2}{2}\rfloor$, and consequently, every closed Riemannian manifold of dimension $n\geq 4$ with positive isotropic curvature has $\pi_i(M)=0$ for all $i\in \{2,\cdots,\lfloor\frac{n}{2}\rfloor\}$. As a consequence, they proved that every closed simply connected Riemannian manifold with positive isotropic curvature is homeomorphic to a sphere.

We note that for $n\geq 4$, the product metric on $S^{n-1}\times S^1$ has positive isotropic curvature, and an equatorial two-sphere contained in $S^{n-1}$ gives a non-constant harmonic map $f:S^2\to S^{n-1}\times S^1$ with Morse index $n-3$. Hence both Theorem \ref{thm} and Corollary \ref{cor:vanishing} are sharp.

We now discuss further topological consequences of Corollary \ref{cor:vanishing}. There are two closely related long-standing conjectures of Gromov and Schoen concerning closed manifolds with positive isotropic curvature:
\begin{conjecture}[Gromov, Schoen]\label{conj:GS}
Let $(M,g)$ be a closed connected Riemannian manifold of dimension $n$ with positive isotropic curvature. Then $\pi_1(M)$ is virtually free, and a finite cover of $M$ is diffeomorphic to either $S^n$ or to a connected sum of finitely many copies of $S^{n-1}\times S^1$.
\end{conjecture}
Combining Corollary \ref{cor:vanishing} with the work of Chodosh--Li--Liokumovich \cite{CLL2023} in dimension five and Mazurowski--Wang--Yao \cite{MWY2026} in dimension six gives Conjecture \ref{conj:GS} in these dimensions with diffeomorphism replaced by homeomorphism. We also use Proposition \ref{prop:intermediate}, relating positive isotropic curvature to the positive intermediate curvature condition introduced by Brendle--Hirsch--Johne \cite{BHJ2024}. Kirby--Siebenmann smoothing theory \cite{KS1977}, together with the work of Friedl--Nagel--Orson--Powell \cite{FNOP2025}, then upgrades the homeomorphism to a diffeomorphism. The five-dimensional case is also summarized in \cite[Theorem~2.1(ii)]{DP2026}.
\begin{corollary}
    Let $n\in \{5,6\}$ and let $(M,g)$ be a closed connected Riemannian manifold of dimension $n$ with positive isotropic curvature. Then $\pi_1(M)$ is virtually free and a finite cover of $M$ is diffeomorphic to either $S^n$ or to a
connected sum of finitely many $S^{n-1}\times S^1$.
\end{corollary}
\begin{proof}
       By Corollary \ref{cor:vanishing}, we have $\pi_2(M)=\cdots=\pi_{n-2}(M)=0$. On the other hand, Proposition \ref{prop:intermediate} implies that $(M,g)$ has positive $4$-intermediate curvature. Thus \cite[Theorem 1]{CLL2023} for $n=5$ and \cite[Theorem 7 (2)]{MWY2026} for $n=6$ imply that $M$ admits a finite cover $\hat M$ homotopy equivalent to $S^n$ or to a connected sum of finitely many copies of $S^{n-1}\times S^1$. In particular, $\hat{M}$ is orientable and $\pi_1(\hat{M})$ is free, so $\pi_1(M)$ is virtually free. By \cite[Theorem 1.3]{GS2009}, $\hat{M}$ is homeomorphic to $S^n$ or to a connected sum of finitely many copies of $S^{n-1}\times S^1$.
       
       Moreover, we have $H^3(\hat{M};\mathbb{Z}/2)=0$. Therefore, the smoothing theory of Kirby--Siebenmann \cite{KS1977} implies  that the PL structure transported across the homeomorphism is isotopic to the standard PL structure. Since smooth structures compatible with a fixed PL structure are unique up to isotopy in dimensions at most six \cite[Theorem~3.5(2)]{FNOP2025}, the transported smooth structure is isotopic to the standard one. Thus $\hat{M}$ is diffeomorphic to $S^n$ or to a connected sum of finitely many copies of $S^{n-1}\times S^1$. This completes the proof.
\end{proof}
Finally, we recall some classification results in dimensions \(n\neq5,6\). Hamilton's seminal work on four-manifolds with positive isotropic curvature \cite{Hamilton1997}, together with subsequent work of Chen--Zhu and Chen--Tang--Zhu \cite{ChenZhu2006,CTZ2012}, led to the complete classification in dimension four, which in particular implies Conjecture \ref{conj:GS} for $n=4$. In higher dimensions, Brendle developed a Ricci flow with surgery theory for manifolds with positive isotropic curvature in dimensions \(n\geq12\) \cite{Brendle2019}, resolving Conjecture \ref{conj:GS} under a mild assumption. Brendle's theorem was recently extended by Chen \cite{Chen2026} to \(n\in \{9,10,11\}\) and very recently by Cho \cite{Cho2026} to \(n\in\{7,8\}\). We also refer to the work of Huang \cite{Huang2025} on removing Brendle's additional assumption in higher dimensions.
\subsection*{Disclosure of AI tools.}
The manuscript was written by the authors. ChatGPT (GPT-5.6 Sol) was used only to suggest the relevance of the smoothing result of \cite{KS1977, FNOP2025, DP2026}. All mathematical arguments, verification, and presentation are the responsibility of the authors.
\section{Proof of Theorem \ref{thm}}
In this section we prove Theorem \ref{thm}. We follow the notation of \cite[\S1.4]{Brendle2010}. Throughout this section, let $f:S^2\to M$ be a non-constant harmonic map. We write 
\begin{align*}
    &\frac{\partial f}{\partial z}=\frac{1}{2}(\frac{\partial f}{\partial x}-i\frac{\partial f}{\partial y}),\\
    &\frac{\partial f}{\partial\bar z}=\frac{1}{2}(\frac{\partial f}{\partial x}+i\frac{\partial f}{\partial y})
\end{align*}
as sections of $f^*(T^{\mathbb{C}}M)$. In particular, since $f$ is harmonic, we have $D_{\frac{\partial}{\partial \bar z}}\frac{\partial f}{\partial z}=0$ and  $\frac{\partial f}{\partial z}$ is a holomorphic section of $f^*(T^{\mathbb{C}}M)$, and $g(\frac{\partial f}{\partial z},\frac{\partial f}{\partial z})=0$.

For every section $V\in \Gamma(f^*(T^{\mathbb{C}}M))$, we write
\begin{align*}
    &D_{\frac{\partial}{\partial z}}V=\frac{1}{2}\big(D_{\frac{\partial}{\partial x}}V-i\,D_{\frac{\partial}{\partial y}}V\big),\\
    &D_{\frac{\partial}{\partial\bar z}}V=\frac{1}{2}\big(D_{\frac{\partial}{\partial x}}V+i\,D_{\frac{\partial}{\partial y}}V\big).
\end{align*}
Next, following the argument of Micallef--Moore \cite{MM1988} (see also \cite[Proposition 1.16]{Brendle2010}), the index form extends complex bilinearly to $I:\Gamma(f^*(T^{\mathbb C}M))\times \Gamma(f^*(T^{\mathbb C}M))\to \mathbb{C}$ and satisfies
\[
I(V,\overline{V})=4\int_{S^2}g(D_{\frac{\partial}{\partial\bar z}}V,D_{\frac{\partial}{\partial z}}\overline{V})\,dx\,dy-4\int_{S^2}R\Big(\frac{\partial f}{\partial z},V,\frac{\partial f}{\partial\bar z},\overline{V}\Big)\,dx\,dy
\]
Let $\mathscr N^{\mathbb{C}}\subset\Gamma(f^*(T^{\mathbb C}M))$ be the
complexification of $\mathscr N$. In particular, the Morse index of $f$ is equal to $\dim_{\mathbb{C}}\mathscr N^{\mathbb{C}}$.
\begin{proposition}\label{prop}
Let $(M,g)$ be a Riemannian manifold of dimension $n\geq 4$ with positive isotropic curvature. Let $E\subset f^*(T^{\mathbb{C}}M)$ be a holomorphic subbundle, and let $\mathscr H\subset \Gamma(E)$ be a complex subspace of holomorphic sections such that $\frac{\partial f}{\partial z}\notin \Gamma(E)$. Then, the Morse index of $f$ is at least $\dim_{\mathbb C}\mathscr H-1$.
\end{proposition}

\begin{proof}
Suppose, to the contrary, that the Morse index of $f$ is at most $\dim_{\mathbb C}\mathscr H-2$. Then there exist two linearly independent sections $V_1,V_2\in \mathscr{H}$ which are $L^2$-orthogonal to $\mathscr N^{\mathbb{C}}$. On the other hand, for $i,j\in \{1,2\}$, we know that $g(V_i,V_j)$ is a holomorphic function on $S^2$, hence a constant. This implies there exists some $s\in \mathbb{C}$, $t\in \mathbb{C}$ with $(s,t)\neq (0,0)$, such that the linear combination $V=s\,V_1+t\,V_2$ satisfies $g(V,V)=0$ at each point of $S^2$. In particular, $V\in\mathscr{H}$ and $V$ is $L^2$-orthogonal to $\mathscr N^{\mathbb{C}}$. Therefore, we must have $I(V,\overline{V})\geq 0$.

By \cite[Proposition 1.17]{Brendle2010}, we have $g(\frac{\partial f}{\partial z},V)=g(\frac{\partial f}{\partial z},\frac{\partial f}{\partial z})=0$. Together with $g(V,V)=0$, the condition that $(M,g)$ has positive isotropic curvature gives $R(\frac{\partial f}{\partial z},V,\frac{\partial f}{\partial\bar z},\overline V)\ge0$, with equality precisely when $\frac{\partial f}{\partial z}$ and $V$ are linearly dependent. This implies $I(V,\overline{V})\leq  0$. Since $I(V,\overline{V})\geq 0$, we must have $\frac{\partial f}{\partial z}$ and $V$ are linearly dependent. As in \cite[Theorem 1.19]{Brendle2010}, we have $V=\psi\,\frac{\partial f}{\partial z}$ for some meromorphic function $\psi:S^2\to\mathbb{C}$. Since $V\in\Gamma(E)$ whereas $\frac{\partial f}{\partial z}\notin\Gamma(E)$, this forces $\psi\equiv0$, contradicting $V\ne0$. 
\end{proof}
We may now complete the proof of Theorem \ref{thm}. As in \cite[Proposition 1.18]{Brendle2010}, there exists a holomorphic subbundle $E\subset f^*(T^{\mathbb C}M)$ such that $\text{\rm rank}_{\mathbb{C}}E\geq n-2$, $c_1(E)=0$ and $\frac{\partial f}{\partial z}\notin\Gamma(E)$. Let $\mathscr H$ be the complex vector space of holomorphic sections of $E$. In partiuclar, the Riemann--Roch theorem implies $\dim_{\mathbb{C}}\mathscr{H}=\text{\rm rank}_{\mathbb{C}}E\mathbin{+}\langle c_1(E),[S^2]\rangle+\dim_{\mathbb{C}}H^1(S^2,E)\geq n-2$, and the Theorem now follows directly from Proposition \ref{prop}.
\appendix

\section{Positive Intermediate Curvature}
Let $(M,g)$ be a Riemannian manifold of dimension $n\geq 4$. We recall the notion of positive intermediate curvature introduced by Brendle--Hirsch--Johne. We fix a positive integer $m\in \{1,\cdots,n-1\}$. We say $(M,g)$ has positive $m$-intermediate curvature if for every $p\in M$, and every orthonormal $n$-frame $\{e_1,\cdots,e_n\}$ at $p$, we have $\sum_{i=1}^m\sum_{j=i+1}^nR_{ijij}>0$. In particular, positive $(n-1)$-intermediate curvature is equivalent to positive scalar curvature.
\begin{proposition}\label{prop:intermediate}
Let $(M,g)$ be a Riemannian manifold of dimension $n\geq 4$ with positive isotropic curvature. Then $(M,g)$ has positive $m$-intermediate curvature for every $m\in \{3,\cdots, n-1\}$.     
\end{proposition}
\begin{proof}
    Let us fix some $p\in M$ and an orthonormal frame $\{e_1,\cdots,e_n\}$ at $p$. For every four distinct indices $i,j,k,\ell$, applying the positive isotropic curvature condition to the four-frames $\{e_i,e_j,e_k,e_{\ell}\}$ and $\{e_i,e_j,e_k,-e_{\ell}\}$ gives
    \begin{equation}\label{eqn:four-curvatures}
    R_{ikik}+R_{i\ell i\ell}+R_{jkjk}+R_{j\ell j\ell}>0.
    \end{equation}
    Clearly, $(M,g)$ has positive scalar curvature by summing over \eqref{eqn:four-curvatures} over every orthonormal four-frames. Hence, we fix some positive integer $m\in \{3,\cdots,n-2\}$. We now consider two cases:

    \textit{Case 1.} Suppose $4\leq m\leq n-2$. Then for every $1\leq i<j<k<\ell\leq m$, applying \eqref{eqn:four-curvatures} to the indices $(i,j,k,\ell)$, $(i,k,j,\ell)$ and $(i,\ell,j,k)$ gives
    \[
    2(R_{ijij}+R_{ikik}+R_{i\ell i\ell}+R_{jkjk}+R_{j\ell j\ell}+R_{k\ell k\ell})>0.
    \]
    Summing over every such $1\leq i<j<k<\ell\leq m$ gives
    \[
        (m-2)(m-3)\sum_{1\leq i<j\leq m}R_{ijij}>0.
    \]
    Moreover, for every $1\leq i<j\leq m$ and $m<k<\ell\leq n$, summing over \eqref{eqn:four-curvatures} for the indices $\{i,j,k,\ell\}$ gives
    \[
        (m-1)(n-m-1)\sum_{1\leq i\leq m,\,m<k\leq n}R_{ikik}>0.
    \]
    The claim then follows from the fact that  
    \[\sum_{i=1}^m\sum_{j=i+1}^nR_{ijij}=\sum_{1\leq i<j\leq m}R_{ijij}+\sum_{1\leq i\leq m,\,m<k\leq n}R_{ikik}.
    \]
    
    \textit{Case 2.} Suppose $m=3$. Then for every $k\in \{4,\cdots,n\}$, applying \eqref{eqn:four-curvatures} to the indices $(1,2,3,k)$, $(1,3,2,k)$ and $(1,k,2,3)$ gives
    \[
    2(R_{1212}+R_{1313}+R_{2323}+R_{1k1k}+R_{2k2k}+R_{3k3k})>0.
    \]
    Summing over $4\leq k\leq n$ gives $$2(n-3)(R_{1212}+R_{1313}+R_{2323})+2\sum_{k=4}^n(R_{1k1k}+R_{2k2k}+R_{3k3k})>0.$$ 
    Moreover, adding \eqref{eqn:four-curvatures} for all $1\leq i<j\leq 3$ and $4\leq k<\ell\leq n$ gives
    \[
    2(n-4)\sum_{k\geq 4}(R_{1k1k}+R_{2k2k}+R_{3k3k})>0.
    \]
    Adding these inequalities gives
    \[
    2(n-3)(R_{1212}+R_{2323}+R_{1313}+\sum_{i=1}^3\sum_{k=4}^nR_{ikik})>0.
    \]
    Finally, since $$\sum_{i=1}^3\sum_{j=i+1}^nR_{ijij}=R_{1212}+R_{1313}+R_{2323}+\sum_{i=1}^3\sum_{k=4}^nR_{ikik},$$ the claim follows directly.
\end{proof}

\end{document}